\documentclass[12pt, a4paper]{article}

\usepackage{amsmath}\multlinegap=0pt
\usepackage[applemac]{inputenc}
\usepackage{amsthm}
\usepackage{amssymb,tikz}
\usepackage[francais,english]{babel}
\numberwithin{equation}{section}
\theoremstyle{plain}
\newtheorem{thm}{Theorem}[section]
\newtheorem{coro}[thm]{Corollary}
\newtheorem{prop}[thm]{Proposition}
\newtheorem{lem}[thm]{Lemma}
\newtheorem{defi}[thm]{Definition}
\newtheorem{remark}[thm]{Remark}

\theoremstyle{definition}

\theoremstyle{remark}

\newtheorem{ex}[thm]{Exemple}
\newcommand{\lan}{\langle}
\newcommand{\ran}{\rangle}
\newcommand{\x}{\mathbf x}
\newcommand{\z}{\mathbf z}
\newcommand{\y}{\mathbf y}

\newcommand{\ess}{\text{\rm{ess}}\sup}

\newcommand{\impl}{\Rightarrow}

\newcommand{\lcal}{\mathcal{L}}

\newcommand{\R}{\mathbb{R}}

\newcommand{\res}{\mathop{\hbox{\vrule height 7pt width .5pt depth 0pt \vrule height .5pt width 6pt depth 0pt}}\nolimits}

\DeclareMathOperator{\spt}{spt}

\newcommand\e{\varepsilon}

\title {Optimal transportation for a quadratic cost with convex constraints and applications}
\date{October 13, 2011}
\author {C.~Jimenez \thanks{\scriptsize\ Laboratoire de Mathématiques de Brest, Université de Bretagne Occidentale, 6, Avenue Victor Le Gorgeu 29200 Brest, FRANCE, {\tt  jimenez@univ-brest.fr} },  F.~Santambrogio
\thanks{\scriptsize\ Laboratoire de Mathématiques d'Orsay, Universit\'e Paris-Sud, 91405 Orsay cedex, FRANCE, \texttt {filippo.santambrogio@math.u-psud.fr} }}

\begin{document}

\maketitle

\begin{abstract}
We prove existence of an optimal transport map in the Monge-Kantorovich problem associated to a cost $c(x,y)$ which is not finite everywhere, but coincides with $|x-y|^2$ if the displacement $y-x$ belongs to a given convex set $C$ and it is $+\infty$ otherwise. The result is proven for $C$ satisfying some technical assumptions allowing any convex body in $\R^2$ and any convex polyhedron in $\R^d$, $d>2$. The tools are inspired by the recent Champion-DePascale-Juutinen technique. Their idea, based on density points and avoiding disintegrations and dual formulations, allowed to deal with $L^\infty$ problems and, later on, with the Monge problem for arbitrary norms.
\end{abstract}

\textbf{Keywords:} optimal transportation, existence of optimal maps, convex bodies, $c-$monotonicity
\bigskip

\textbf{Acknowledgments:} The authors want to thank the French-Italian Galilée Project ``Allocation et Exploitation et Evolution Optimales des Ressources: r\'eseaux, points et densit\'es, mod\`eles discrets et continus'' and the support of the ANR project OTARIE (ANR-07-BLAN-0235 OTARIE).

\clearpage

\section{Introduction}
The optimal transport problem, introduced by G. Monge in \cite{monge} at the end of the 18th century, has become very famous in the past 20 years, when the relaxed formulation given by L.V. Kantorovich in the '40s, coupled with the most recent advances, has finally allowed to give a complete solution to many useful issues of the problem. Given $f_0<<dx$ and $f_1$ two probability measures on $\R^d$, the issue is to push forward $f_0$ to $f_1$ with a map $T:\R^d\rightarrow \R^d$ called optimal transport map which solves the following optimization problem:
$$(M)\quad \inf\left\{\int_{\R^d} c(x,T(x))\ df_0(x):\ T_\sharp f_0=f_1\right\}$$
where $c:\R^d\times \R^d\rightarrow [0,+\infty]$ is a fixed cost function. The original problem by G. Monge concerned the case where $c$ is the Euclidian distance, but later on many other cost functions have been investigated, and in particular the quadratic one $c(x,y)=|x-y|^2$, which has been popularized by \cite{bre}. We will not enter now into the details of the formulation that L.V. Kantorovich gave to overcome the difficulties of the nonlinear behavior of Problem (M) since we want to present as soon as possible the particular cost function that we want to consider. In this paper we will actually focus on a quadratic cost with a closed convex constraint $C\subset \R^d$:
\begin{equation}\label{cost}
c(x,y)=
\left\{\begin{array}{l c}
\vert y-x \vert^2 & \mbox{ if } y-x\in C\\
+\infty &\mbox{ otherwise, }
\end{array}
\right.
\end{equation}
where we denote by $\vert \cdot \vert$ the Euclidian norm.\\
Solving $(M)$ for this cost seems a natural question, both for its modelization meaning (imagine that $C$ is the unit ball in $\R^d$: in such a case the problem aims at finding the optimal displacement minimizing the total effort, subject to a pointwise constraint $|T(x)-x|\leq 1$) and for mathematical applications. Actually, it appears in \cite{CarDePSan} that a strategy to solve more general transport problems passes  through the solutions of some constrained cases, where it is imposed that $T(x)-x$ belongs to a given set, which is often the face of a convex body. The $L^\infty$ problems studied in \cite{Winfty} can also be considered constrained problems (if the minimal $L^\infty$ norm $||T(x)-x||_{L^\infty}$ is $L$, this means that we could consider maps constrained to satisfy $T(x)-x\in \overline{B(0,L)}$). In many of these cases, once the problem is transformed into a convexly constrained one, then it is only necessary to find an arbitrary transport map satisfying such a constraint. The minimization of the quadratic cost among them is just an additional criterion to select a special one, and it is chosen since it usually guarantees better properties. We refer anyway to  \cite{CarDePSan} for a more detailed discussion about convex costs in general and quadratic costs with convex constraints.

To illustrate the difference between our setting and the usual quadratic case $c(x,y)=\vert x-y\vert^2$, we give an example.\\
\begin{ex} \label{ex}
Let $u_i$, with $i=1,2,3$ three points of the plan with the following distances: $\vert u_1-u_2\vert=\vert u_2-u_3\vert=1$ and $\vert u_1-u_3\vert=5/4$. Take 
$B(u_i,\e)$ with $\e>0$  and $i=1,2,3$ three tiny balls and set $f_0$ and $f_1$ absolutely continuous with respect to the Lebesgue measure with densities $1_{B(u_1,\e)}+1_{B(u_2,\e)}$ and $1_{B(u_2,\e)}+1_{B(u_3,\e)}$ respectively. The unique optimal transport plan for the usual quadratic cost is $T_2(x)=x+(u_3-u_1)$ if $x\in B(u_1,\e)$ and the identity elsewhere, but this map is not admissible for the cost $c$ defined by (\ref{cost}) when $C$ is the unit ball.  The unique optimal transport plan for this cost is hence $T_c(x)=x+(u_{i+1}-u_i)$ if $x\in B(u_i,\e)$ for $i=1,2$. 
\end{ex}

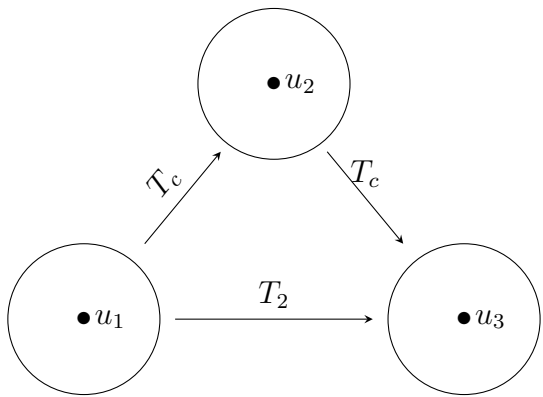
\begin{figure}[h]
\begin{center}

\begin{tikzpicture}
\draw  (-2.5,0) circle (1) ;
\draw  (2.5,0) circle (1) ;
\draw  (0,3.12) circle (1) ;
%
%\fill[white]  (1.732,1) circle (2) ;
%\fill[white] (0.366,-2.336) circle (1.931) ;
%
\draw  (-2.5,0)  node[right] {$u_1$} node{$\bullet$};
\draw  (2.5,0)  node[right] {$u_3$} node{$\bullet$};
\draw  (0,3.12)  node[right] {$u_2$} node{$\bullet$};

\draw [->,>=stealth]  (-1.3,0)  --  (1.3,0) node[midway,above] {$T_2$};
\draw [->,>=stealth]  (-1.7,1)  --  (-0.7,2.22) node[midway,above,sloped] {$T_c$};
\draw [->,>=stealth]  (0.7,2.22)  --  (1.7,1) node[midway,above] {$T_c$};

\end{tikzpicture}
\end{center}
\caption{The situation in Example \ref{ex}}
\label{loc1_fig1}
\end{figure}

As we said, the quadratic cost has often better properties than other general cost functions, and it is not astonishing that the first existence result for $(M)$ was actually proven for $c(x,y)=\vert y-x\vert^2$.  Actually, Y. Brenier proved in 1991, existence and uniqueness of an optimal map $T$, and the crucial point in the proof is the existence of a dual formulation 
$$\sup_{(u,v)\in {\cal C}(\R^d)\times {\cal C}(\R^d)}\left\{ \int_{\R^d} v(y)\ df_1(y)-\int_{\R^d} u(x)\ df_0(x):\ v(y)-u(x)\le c(x,y)\right\},$$
where the $\sup$ is actually achieved. This allowed  Brenier to use the primal-dual optimality condition and prove the optimality of the map
$$T(x)=x-{1\over 2}\nabla u(x).$$ 
This scheme of proof was successfully extended  to the cases $c(x,y)=\vert x-y\vert^p$ with $p>1$ (or, more generally, to any strictly convex function of $x-y$) by W. Gangbo and R.J McCann \cite{GM}, and by L.A Caffarelli \cite{Caf}. The original problem by Monge ($p=1$) happens to be more difficult and needed other arguments and techniques. It was solved later (see L.C. Evans and W. Gangbo \cite{EG}, N.S. Trudinger and X.J. Wang \cite{TW}, L. Ambrosio and A. Pratelli \cite{AP}, M. Feldman R.J. McCann \cite{FM}). The existence for the dual problem was also crucial in these proofs.\\ 

When we add a constraint, the existence of an optimal pair of functions  for the dual formulation $(D)$ is not guaranteed: the usual proofs of this result require continuity or finiteness of the cost (see the books of C. Villani \cite{villani} and \cite{V2}). Some very irregular pairs of solutions (not even functions) were found for non finite Borel costs by M. Beiglb\"ock, C. L\'eonard  and W. Schachermayer \cite{BLS}. Let us mention also that under the additional restrictive assumption $\int c(x,y)\ df_0(x)df_1(y)<+\infty$, W. Schachermayer and J. Teichmann proved the existence of an optimal pair of Borel functions for $(D)$ \cite{ST}. Hence, though very interesting, these results cannot be used in our setting.\\

Nevertheless, more recently, T. Champion, L. De Pascale and P. Jutinenn developed a completely different idea and manage to solve the `` $L^{\infty}$ case'' (see \cite{Winfty}). It is the case where, instead of minimizing the integral $\int |T(x)-x|^pdf_0(x)$, i.e. the $L^p$ norm, they considered the minimization of the $L^\infty$ norm (as we mentioned before).  This frame is very peculiar because, due to its non-integral formulation, there is no dual formulation for the Monge problem. Hence, they developed a completely different idea, which was powerful enough to be extended later on to some cases where the cost is still integral, but for a function $c(x,y)$ which is not strictly convex in $x-y$. In particular, this gave the proof of the existence of a solution $T$ to Monge's Problem for $c(x,y)=||x-y||$ for a general norm (\cite{CD1} and \cite{CD2}), which was unsolved before.\\
In this paper, we show that we can solve $(M)$ with $c$ defined by (\ref{cost}) using the same tools as T. Champion, L. De Pascale and P. Jutinenn.

In the next section, we state the main results. In section 3, we make some comments and give a proof in dimension 1. The general proof is given in subsection 4.2. after proving some preliminary results in subsection 4.1. Finally, in the last section, we give an application of our result to the case of crystalline norms and to some $L^\infty$ problems. 

\section{Main result}
In all the following, the cost $c$ is defined by (\ref{cost}), $f_0$ and $f_1$ are two probability measures on $\R^d$ and $f_0$ is absolutely continuous with respect to the Lebesgue measure.
Let us introduce the relaxed formulation of the Monge problem introduced by Kantorovich. This new problem, $(MK)$, satisfies $\inf(MK)=\inf(M)$ (see \cite{villani} and \cite{V2}):
$$(MK)\quad \inf\left\{\int_{\R^d\times\R^d} c(x,y)\ d\pi(x,y):\ \pi\in \Pi(f_0,f_1)\right\}$$
where $\Pi(f_0,f_1)$ is the set of probability measure on $\R^d\times\R^d$ with fixed marginals $f_0$ and $f_1$. The following classical result holds:

\begin{prop}
If there exists $\pi\in \Pi(x,y)$ such that $ \int_{\R^d\times\R^d} c(x,y)\ d\pi(x,y)$ is finite, then $(MK)$ has a solution.
\end{prop}

The classical strategy to solve $(M)$ is to show that at least a solution $\gamma$ of $(MK)$ is concentrated on the graph of a map $T$ i.e:
$$\int \varphi(x,y)\ d\gamma(x,y)=\int \varphi(x,T(x))\ df_0(x)\quad \forall \varphi\in {\cal C}_c(\R^d).$$ 
Indeed, the map $T$ is admissible for $(M)$ if $\gamma$ is admissible for $(MK)$ and it is optimal if $\gamma$ is optimal. Conversely for any $T$ admissible for $(M)$, $\gamma=({\rm Id}\times T)_\sharp f_0$ is admissible for $(MK)$, and if $T$ is optimal, $\gamma$ is optimal too.
\medskip

In order to state the main theorem of this paper, we need to give some definitions on the geometry of convex sets.
\begin{defi} Consider a convex set $C\subset\R^d$.
\smallskip
\begin{enumerate}
\item If $x_0$ is any point of $C$, we call {\bf dimension} of $C$ the dimension of ${{\rm Span}(C-x_0)}$, i.e. the dimension of the smallest affine space containing $C$. This number is denoted $\dim(C)$ and does not depend on the choice of $x_0$. The  interior and boundary of $C$ in the canonical topology associated to the Euclidian distance in such an affine space are called relative interior and {\bf relative boundary} and denoted by $ {\rm ri}C$ and  ${\rm r}\partial C$.
\item A convex subset $F\subset C$ containing more than one point is called a {\bf straight part} of $C$ if it is contained in relative boundary ${\rm r}\partial C$.
\item A subset $F$ of a convex set $C$ is called a {\bf maximal flat part}  of $C$ if it is maximal for the inclusion among straight parts of $C$.
\item A subset $F$ of a convex set $C$ is called a {\bf flat part}  of $C$ if it there is a finite chain of convex sets $F=F_1\subset F_2\subset\dots \subset F_{k-1}\subset F_k=C$ ($k\geq 2$) such that, for each $i=1,\dots, k-1$, $F_i$ is a maximal flat part of $F_{i+1}$. Notice that in this case one has $\dim(F_1)<\dim(F_2)<\dots<\dim(F_k)$. In particular, for every flat part $F$ of $C$ we have $1\leq \dim(F)\leq (d-1)$. 
\item A set $C\subset \R^d$ is said to be {\bf strictly convex} if it has no flat parts.
\end{enumerate}
\end{defi}

The aim of this paper is to prove the following result:
\begin{thm}\label{THM}
Assume that:
\begin{enumerate}
\item[i)] there exists $\pi\in \Pi(f_0,f_1)$ such that $ \int_{\R^d\times\R^d} c(x,y)\ d\pi(x,y)<+\infty$,
\item[ii)] $C$ has at most a countable number of flat parts.
\end{enumerate}
Then, if $f_0$ is absolutely continuous with respect to the Lebesgue measure on $\R^d$,  any solution $\gamma$ of $(MK)$  is concentrated on the graph of a map $T$ which is optimal for $(M)$.
\end{thm}

\begin{remark}
Assumption ii) is obviously satisfied if $C$ is strictly convex. Moreover, if $C$ satisfies ii), then any flat part of $C$ is a convex subset of an affine space $\R^k$ for $k<d$ which satisfies the same assumption. 
Finally, this assumption allows every bounded convex set in $\R^2$ (because flat parts should be subsets of positive length of its boundary, which has finite perimeter, and hence they cannot be more than countably many) and all compact polytopes in any dimension.
\end{remark}

\begin{coro}
Under the assumption of Theorem \ref{THM}, the solutions of $(MK)$  and $(M)$ are unique. 
\end{coro}

\begin{proof} [Proof of the corollary] The idea is classical : let $\gamma_1$ and $\gamma_2$ be two solutions of $(MK)$, let $T_1$ and $T_2$ the associated transport maps (given by Theorem \ref{THM}). The probability measure ${\gamma_1 +\gamma_2\over 2}$ is also associated to a transport map $T_3$. Then for any $\varphi\in {\cal C}_c(\R^d\times \R^d)$:
$${1/2}\left(\int_{\R^d} \varphi(x,T_1(x))\ df_0(x) +\int_{\R^d} \varphi(x,T_2(x))\ df_0(x)\right)= \int_{\R^d} \varphi(x,T_3(x))\ df_0(x).$$
This implies $f_0(\{x:\ T_1(x)\not=T_2(x)\})=0$, that is $\gamma_1=\gamma_2$ and $T_1=T_2$ almost everywhere. 
\end{proof}

\section{Proof and comments in dimension 1}
The one-dimensional case is a very special one. Actually, for a very wide class of transport costs, the optimal transport is always characterized in the same way and is always the same, namely the unique nondecreasing transport map sending $f_0$ to $f_1$. This is summarized in the following theorem (the proof is written here in its full generality, even if it is essentially taken from \cite{ambrosio}):
\begin{thm}
Let $\phi:\R\to\R\cup\{+\infty\}$ be a l.s.c. strictly convex function (which actually means strictly convex on its domain $\{\phi<+\infty\}$), and $f_0$ and $f_1$ two probability measures on $\R$, with $f_0$ non-atomic. Then the optimal transport problem 
$$\inf\left\{\int_{\R^2} \phi(y-x)\ d\pi(x,y):\ \pi\in \Pi(f_0,f_1)\right\}$$
has a unique solution, which is given by $\gamma=({\rm Id}\times T)_\# f_0$, where $T$ is the nondecreasing transport map from $f_0$ to $f_1$, which is well defined $f_0-$a.e.

Moreover, if the strict convexity assumption is dropped and $\phi$ is only convex, then the same $T$ is actually an optimal transport map, but no uniqueness is guaranteed anymore.
\end{thm}
\begin{proof}
We will use the usual strategy based on $c-$cyclical monotonicity, i.e. the fact that any optimal $\gamma$ is concentrated on a cyclical monotone set $\Gamma$, where we can also assume $\phi(y-x)<+\infty$. This means that $(x,y),(x',y')\in\Gamma$ implies
\begin{equation}\label{c-c}
\phi(y-x)+\phi(y'-x')\leq \phi(y'-x)+\phi(y-x').
\end{equation}
We only need to show that this implies (in the strictly convex case) a monotone behavior : we will actually deduce from \eqref{c-c} that $x<x'$ implies $y\leq y'$. This allows to say that $\Gamma$ is included in the graph of a monotone (nondecreasing) multifunction. Hence, up to a countable set, for every $x$ there is unique image $y=T(x)$, and the map $T$ is nondecreasing.

To prove $y\leq y'$ suppose by contradiction $y>y'$ and denote $a=y-x$, $b=y'-x'$ and $c=x'-x$. 
Condition \eqref{c-c} reads as $\phi(a)+\phi(b)\leq \phi(b+c)+\phi(a-c)$. Moreover, we have
$$b+c=(1-t)b+ta,\quad a-c=tb+(1-t)a,\quad \mbox{ for }t=\frac{c}{a-b}.$$
The assumption $y'<y$ reads as $b+c<a$, which gives $t\in]0,1[$ (since it implies $b<a$, and hence $t>0$, and $c<a-b$, and hence $t<1$). Thus, convexity yields
\begin{eqnarray*}
\phi(a)+\phi(b)&\leq& \phi(b+c)+\phi(a-c)\\
&\leq& (1-t)\phi(b)+t\phi(a)+t\phi(b)+(1-t)\phi(a)=\phi(a)+\phi(b).
\end{eqnarray*}
Since we assumed $\phi(a),\phi(b)<+\infty$, we also deduce  $\phi(b+c),\phi(a-c)<+\infty$ and the strict convexity implies a strict inequality, thus getting to a contradiction.

This shows that a strictly convex cost always admits a unique optimal transport plan, which is given by the monotone nondecreasing transport map. The statement when $\phi$ is only convex is obtained by approximation through $\phi(y-x)+\varepsilon |y-x|^2$ (or through any other strictly convex approximation).
\end{proof}

It is important to notice that in this one-dimensional convex framework the optimal transport does not really depend on the cost. This is very useful for approximation procedures (if, for instance, one considers costs of the form $|y-x|^p$ and lets $p\to\infty$ it is possible to deduce the optimality of the same nondecreasing transport for the $L^\infty$ transportation problem, see again \cite{Winfty}). Not only, in our case when $\phi(z)=|z|^2$ for $z\in C$ and $\phi=+\infty$ outside $C$, this shows that the optimal transport is not really affected by the constraint (differently from the example in $\R^2$ that we gave in the introduction). What is affected by the constraint is the possibility of transporting $f_0$ onto $f_1$ {\it at finite cost}. Depending on the two measures, it is possible that the nondecreasing map $T$ satisfies $T(x)-x\in C$ or not. If yes, then it is optimal; if not, then the problem has no solution with a finite cost.

\section{Proof of Theorem \ref{THM}}
\subsection{Preliminary results}
Let $\gamma$ be an optimal transport map for $(MK)$.  We recall that when the cost is l.s.c. (which is the case of our cost function $c$), then any optimal $\gamma$ is concentrated on a $c$-cyclically monotone set $\Gamma$.  This means in particular
$$(x_0,y_0),(x_1,y_1)\in\Gamma\impl c(x_0,y_0)+c(x_1,y_1)\leq c(x_1,y_0)+c(x_0,y_1)$$
(we refer again to \cite{ambrosio} for the actual definition of $c-$cyclically monotone sets, which imposes a similar condition for families of $k$ points in $\Gamma$ and arbitrary permutations of their second coordinates; here we only need the condition for $k=2$). In our case, thanks to the definition of $c$, if we use the equivalence
$$|x_0-y_0|^2+|x_1-y_1|^2\leq  |x_0-y_1|^2+|x_1-y_0|^2\Leftrightarrow \lan x_1-x_0\,,\,y_1-y_0\ran\geq 0$$
then we get
\begin{equation}\label{monotone} 
\left.
\begin{array} {c}
\mbox{ If $(x_0,y_0),\ (x_1,y_1)\in \Gamma$ are such that:}\\
y_0-x_1\in C\mbox{ and }y_1-x_0\in C\\ 
\end{array}
\right\}\Rightarrow
\lan y_1-y_0,x_1-x_0 \ran\ge 0.
\end{equation}
Without loss of generality, we may assume that $\Gamma$ is $\sigma$-compact (which is useful for measurability reasons) and that any $(x,y)\in \Gamma$ is such that $y-x\in C$.\\
Arguing as in \cite{Winfty},  we introduce the following object:
\begin{defi}
For any $r>0$ and any $y\in \R^d$, we set:
$$\Gamma^{-1}(\overline{B(y,r)}):=\pi_x(\Gamma\cap (\R^d \times \overline{B(y,r)}))$$
where $B(y,r)$ denotes the open ball of radius $r>0$ centered at $y$. The assumption of $\sigma-$compactness on $\Gamma$ guarantees that the set that we just defined is measurable.
\end{defi}

We recall the following results extracted from \cite{Winfty} and \cite{CD1}:
\begin{lem}\label{lemCD}
The measure $\gamma$ is concentrated on a $\sigma$-compact set $R(\Gamma)$ such that for all $(x_0,y_0)$ in $R(\Gamma)$ and for all $r>0$, the point $x_0$ is a Lebesgue point of $\Gamma^{-1}(B(y_0,r))$, i.e.:
$$\lim_{\e\to 0} {\vert \Gamma^{-1}(\overline{B(y_0,r)})\cap B(x_0,\e)\vert \over \vert  B(x_0,\e)\vert}=1.$$
\end{lem}

\begin{prop}\label{PropGT}
Let $(x_0,y_0)\in R(\Gamma)$, let $r>0$, $\delta\in ]0,1[$ and $\xi$ a unit vector in $\R^d$, then for  any $\e>0$  the following set has a positive Lebesgue measure:
$$\Gamma^{-1}(\overline{B(y_0,r)})\cap B(x_0,\e)\cap C(x_0,\xi,\delta),$$
where the set  $C(x_0,\xi,\delta)$ is the following convex cone:
$$\{x:\ \lan x-x_0, \xi\ran\ >(1-\delta)\vert x-x_0\vert\}$$ 
(notice that it is an open cone, with $x_0\notin  C(x_0,\xi,\delta)$).
\end{prop}
\begin{lem}\label{interieur1}
Let $x\in C$, $\xi$ a unit vector $\eta>0$ such that $x+\eta\xi\in {\rm Int}(C)$. Let $0<\eta'<\eta$ be fixed.
Then there exist $r>0$ and $\delta\in ]0,1[$  such that:
$$\forall\ y\in B(x,r)\cap C,\quad \left(C(y,\xi,\delta)\cap \overline{ B(y,\eta')}\right)\subset {\rm Int}(C).$$
\end{lem}

\noindent The direction $\xi$ is called an entering direction at the point $x$.\\
\begin{proof}  First of all notice that, if $x+\eta\xi\in {\rm Int}(C)$, then all the interior points of the segment connecting $x$ to $x+\eta\xi$ must belong to $ {\rm Int}(C)$.

We prove the existence of $r>0$ and and $\delta\in ]0,1[$  such that:
\begin{equation}\label{couronne}
\forall\ y\in B(x,r)\cap C,\quad \left(C(y,\xi,\delta)\cap\left[ \overline{ B(y,\eta')}\setminus B(y,\eta'/2)\right]\right)\subset {\rm Int}(C).
\end{equation}
If we prove \eqref{couronne}, the result then follows by passing to the convex envelope. 
The thesis (\ref{couronne}) may be proven by contradiction. Suppose that it is false, and find a sequence $y_n\in C$ such that $y_n\to x$, together with a sequence $z_n\notin {\rm Int}(C)$ and $z_n\in \left(C(y_n,\xi,\delta_n)\cap\left[\overline{ B(y_n,\eta')}\setminus B(y_n,\eta'/2)\right]\right)$, for $\delta_n\to 0$.
Up to subsequence we may suppose $z_n\to z$. This limit point $z$ must satisfy  $z\notin {\rm Int}(C)$ (since $ {\rm Int}(C)$ is open) and $|z-x|\in [\eta'/2,\eta']$. Moreover, from $\lan z_n-y_n,\xi\ran\geq (1-\delta_n)|z_n-y_n|$, we get $\lan z-x,\xi\ran\geq|z-x|$, i.e. $z-x=\lambda\xi$ for $\lambda\geq 0$. Then, we get $\lambda\in  [\eta'/2,\eta']$ and finally we obtain a contradiction since $z$ belongs to the interior of the segment connecting $x$ to $x+\eta\xi$, but not  to ${\rm Int}(C)$.
\end{proof}

\subsection{The proof in dimension $d>1$}
We start proving a fundamental lemma:
\begin{lem}\label{central}
Let $(x_0,y_0)$ and $(x_0,y_0')$ two points of $R(\Gamma)$. We assume
$$\{(1-t)(y_0-x_0)+t(y_0'-x_0):\ t\in [0,1]\}\cap {\rm Int}(C)\not =\emptyset.$$ 
Then $y_0=y_0'$.
\end{lem}

\begin{proof}
To show this lemma, we proceed by contradiction and adapt the method developed in \cite{Winfty} to our case.
 Let $(x_0,y_0)$ and $(x_0,y_0')$ two points of $R(\Gamma)$ satisfying the lemma's assumption with $y_0\not =y_0'$. \\
We set  $\xi:={y_0'-y_0\over \vert y_0'-y_0\vert}$. The direction $\xi$ is an entering direction at $y_0-x_0$ and $-\xi$ is an entering direction at $y_0'-x_0$.
Using  lemma \ref{interieur1}, we take  $r>0$, $\rho \in]0,1[$ and $\e>0$, such that:
\begin{equation}\label{blue2}
\left(C(y'_0-x_0,-\xi,\rho)\cap \overline{B(y_0'-x_0,\e)}\right)\subset {\rm Int}(C),
\end{equation}
\begin{equation}\label{pink}
\forall z\in B(y_0-x_0,r+\e),\quad \left(C(z,\xi,\rho)\cap \overline{B(z,\e)}\right)\subset {\rm Int}(C). 
\end{equation}
Moreover, if $\rho$ and $r$ are small enough, we also have %
\begin{equation}\label{Decroissance}
\forall x\in C(x_0,\xi,\rho),\ \forall y\in \overline{B(y_0,r)}:\quad \lan x-x_0,y-y_0' \ran <0.
\end{equation} 
Indeed the direction of $x-x_0$ is almost that of $\xi$ (up to an error $\rho$) and that of $y-y_0'$ almost $-\xi$ (up to an error of order $r/|y_0-y'_0|$). The two vectors being opposite, it is clear that if both $\delta$ and $r$ are small then the scalar product stays negative.
%%
%\begin{eqnarray*}
%\lan x-x_0,y-y'_0 \ran&=&\lan x-x_0,y-y_0 \ran + \lan x-x_0,y_0-y'_0 \ran \\
%&\le & r\vert x-x_0\vert - (1-\rho)  \vert x-x_0\vert \vert y_0-y'_0 \vert<0.
%\end{eqnarray*}
%
\noindent By use of Proposition \ref{PropGT}, we 
get the existence of a couple $(x,y)\in \Gamma$ such that $x\in \overline {B(x_0,\e)}\cap C(x_0,\xi,\rho)$ and $y\in \overline{B(y_0,r)}$.\\
Then the couple $(x,y)\in \Gamma$ satisfies:
\begin{itemize}
\item $y_0'-x=y_0'-x_0+(x_0-x)\in C$, as a consequence of (\ref{blue2}),
\item  $\lan x-x_0,y-y_0' \ran <0,$ by \eqref{Decroissance}.
\end{itemize}
To get a contradiction with (\ref{monotone}), it remains only to show $y-x_0\in C$.
This comes from (\ref{pink}). Indeed $y-x=(y_0-x_0)+(x_0-x)+(y-y_0)$ is in $B(y_0-x_0,r+\e)$ and $x-x_0\in (C(0,\xi,\rho)\cap \overline{B(0,\e)})$ so that
$$y-x_0=(y-x)+(x-x_0)\in {\rm Int} C.\qedhere$$ 
\end{proof}
We easily deduce
\begin{coro}
If $C$ is strictly convex, then any solution of $(MK)$ is supported on a graph. \\
Moreover, if $C$ is not strictly convex, any solution of $(MK)$ is supported on a set $R(\Gamma)$ satisfying the following property:\\
If $(x_0,y_0)$ and $(x_0,y_0')$ are in $R(\Gamma)$, then $(y_0-x_0)$ and $(y_0'-x_0)$ belong to a same flat part of $C$.
\end{coro}

To prove Theorem \ref{THM} itself, we proceed by induction on the dimension $d$ of the space. The theorem in dimension $d=1$ has already been  proved. Let us show that if the theorem holds in any dimension lower or equal to $d-1$, it also holds in dimension $d$.\\

\noindent Denote by $C_0$ the set of points of $C$ which do not belong to any flat part of $C$, and by $(C_i)_{i\in I}$ the family of the  maximal flat parts of $C$. Let $\gamma$ an optimal transport plan and $R(\Gamma)$ defined as previously so that $\gamma$ is concentrated on $R(\Gamma)$. We set 
$$\Gamma_i:=\{(x,y):\ y-x\in C_i\}\cap R(\Gamma),\  \forall i\in I\cup\{0\}.$$
By Lemma \ref{central}, the measure $\gamma\res \Gamma_0$ is concentrated on a graph, and, for any $i\in I$, we have the following property:
$$(x,y)\in \Gamma_i\Rightarrow (\{(x,y'):\quad (x,y')\in R(\Gamma)\}\subset  \Gamma_i).$$
Then it is enough to show that for  any $i\in I$, $\gamma_i:=\gamma\res \Gamma_i$ is concentrated on a graph and to glue the results together. Each measure $\gamma_i$ is an optimal transport plan for the same cost $c$ between its marginals. The assumption that $I$ is at most countable allows to decompose $\gamma$ into a sum of these measures, thus avoiding the need for disintegration results. In particular, it is clear that the key assumption that $f_0$ is absolutely continuous stays true for the first marginal of $\gamma_i$ (which would not be easy in case of a more-than-countable disintegration).
%It is easily seen that for all $i\in I$, the measure $\gamma\res \Gamma_i$ is an optimal transport plan between its $x$ and $y$ marginals %for the cost $c$.

\noindent Let us fix $i\in I$, let $n\le d-1$ be the dimension of $C_i$ and $z$ a fixed vector in $C_i$. Up to a change of variables, we may assume 
$${\rm Span}(C_i-z)=  {\rm Span}(e^1,....e^n)$$
where ${(e^1,...,e^d)}$ is the canonic base of $\R^d$. \\
{\bf Notations:}
We denote:
$${\x_1:=(x^1,...,x^n)},\   {\x_{2}:=(x^{n+1},...,x^d)},\     \z_2=(z^{n+1},...,z^d).$$
$$\hat{\gamma}:=\gamma\res \Gamma_i,\        \hat f_0:=\pi_x\sharp \hat\gamma,\   \hat f_1:=\pi_y\hat \gamma.$$ 
Note that $y-x\in C_i$ implies $\y_2-\x_2=\z_2$ so $\hat{\gamma}$ is concentrated on $$\{(x, (\y_1,\x_2+\z_2)):\ x=(\x_1,\x_2)\in\R^d,\ \y_1\in \R^d\}.$$
Making a disintegration of $\hat\gamma$, $\hat f_0$, $\hat f_1$ we write:
$$\hat\gamma(x,(\y_1,\x_2+\z_2))=\hat \mu^{\x_2}(\x_1,\y_1)\otimes\zeta(\x_2,\y_2),\quad \mbox{ with }\zeta:=\pi_{\x_2,\y_2}\sharp\hat \gamma,  $$
$$\hat f_0=\hat f_0^{\x_2}(\x_1) \otimes\left(\pi_{\x_2}\sharp\hat f_0\right)(\x_2),\  \hat f_1=\hat f_1^{\y_2}(\y_1) \otimes\left(\pi_{\y_2}\sharp\hat f_1\right)(\y_2).$$
Notice moreover that the measure $\zeta$ is concentrated on the pairs of the form $(\x_2,\x_2+\z_2)$. Then we have:
\begin{eqnarray*}
\int c(x,y)\ d\hat \gamma(x,y)&=& \int \vert \x_1-\y_1\vert^2\ d\hat \gamma(x,y)+ \vert \z_2\vert^2\gamma(\Gamma_i)\\
&=&\int \vert \x_1-\y_1\vert^2\ d\hat \mu^{\x_2}(\x_1,\y_1)d\zeta(\x_2,\x_2+\z_2) + \vert \z_2\vert^2\gamma(\Gamma_i).
\end{eqnarray*}
We set: $\hat c(\x_1,\y_1)= \left\{\begin{array}{l l}
\vert \x_1-\y_1\vert^2&\mbox{ if } (\y_1-\x_1,\z_2)\in C_i, \\
+\infty & \mbox{ otherwise.}
\end{array}
\right.$

\begin{lem}
$\zeta$- almost every $(\x_2,\x_2+\z_2)$, we have:
\begin{itemize}
\item[i)] $\hat f_0^{\x_2}= \pi_{\x_1}\sharp \hat \mu^{\x_2},\mbox{ and }  \hat f_1^{\y_2}= \pi_{\y_1}\sharp \hat \mu^{\x_2},$
moreover  $\hat f_0^{\x_2}$ is absolutely continuous with respect to the Lebesgue measure,
\item[ii)] the measure $\hat \mu^{\x_2}$ on $\R^d\times\R^d$ is a solution of the following 
transport problem:
$$(P^{\x_2})\quad\min\left\{\int  \hat c(\x_1,\y_1)\ d\pi(\x_1,\y_1):\     \pi\in \Pi(\hat f_0^{\x_2},\hat f_1^{\y_2})\right\}.$$
\end{itemize}
\end{lem}
\noindent Then using the theorem \ref{THM} in dimension $n\le d-1$, we get that, for $\zeta$- almost every $(\x_2,\x_2+\z_2)$ the measure $\hat\mu^{\x_2}$ is concentrated on a graph. Then, recalling that $\zeta$ is concentrated on $\{(\x_2,\x_2+\z_2)\}$, we get the desired result.\\

\begin{proof} [Proof of the lemma]
i) Let us show $\hat f_0^{\x_2}= \pi_{\x_1}\sharp \hat \mu^{\x_2}$. On the one hand, we have:
\begin{eqnarray*}
\hat f_0&=& \pi_{\x_1,\x_2}\sharp \hat \gamma= \pi_{\x_1,\x_2}\sharp (\hat \mu^{\x_2}(\x_1,\y_1)\otimes \zeta(\x_2,\y_2))\\
&=&(\pi_{\x_1}\sharp \hat \mu^{\x_2})\otimes (\pi_{\x_2}\sharp\zeta(\x_2,\y_2)).
\end{eqnarray*}
On the other hand: $$\pi_{\x_2}\sharp\zeta(\x_2,\y_2)=\pi_{\x_2}\sharp\pi_{\x_2,\y_2}\sharp\hat \gamma(\x,\y)=\pi_{\x_2}\sharp \hat f_0.$$
Then, by uniqueness of the disintegration $\hat f_0^{\x_2} $ of $\hat f_0$  outside a $\pi_{\x_2}\sharp f_0$ negligible set, we get the result.\\
In the same way, we can show $\hat f_1^{\y_2}= \pi_{\y_1}\sharp \hat \mu^{\x_2}$. Moreover as $f_0$ is absolutely continuous with respect to Lebesgue, also is $\hat f_0^{\x_2}$ (with respect to the $n-$dimensional Lebesgue measure, instead of the $d-$dimensional one).

\noindent ii) We proceed by contradiction. Assume there exists $E\subset \{(\x_2,\x_2+\z_2)\}$ not negligible such that, for all  $(\x_2,\x_2+\z_2)\in E$,  the measure $\hat\mu^{\x_2}$ is not optimal for 
$(P^{\x_2})$. For any  $(\x_2,\x_2+\z_2)\in E$, let us pick another measure $\bar \mu^{\x_2}$ chosen so as to be optimal for  $(P^{\x_2})$. We set:
$$\bar{\gamma}=\bar\mu^{\x_2}\otimes \zeta\res E+  \hat\mu^{\x_2}\otimes \zeta\res {E^c}$$
where $E^c$ is the complementary of  $E$ in $\tilde C_i$. \\
Then $\bar{\gamma}\in \Pi(\hat f_0,\hat f_1)$,  indeed:
\begin{eqnarray*}
\pi_x\sharp \bar \gamma&=&  \pi_x\sharp (\bar\mu^{\x_2}\otimes \zeta\res E)+   \pi_x\sharp (\hat\mu^{\x_2}\otimes \zeta\res {E^c})\\
&=&  \pi_{\x_1}\sharp \bar\mu^{\x_2}\otimes \pi_{\x_2}\sharp\zeta\res E+   \pi_{\x_1}\sharp \hat\mu^{\x_2}\otimes  \pi_{\x_2}\sharp\zeta\res {E^c}\\
&=& \hat f_0^{\x_2}\otimes ( \pi_{\x_2}\zeta\res E+ \pi_{\x_2}\sharp\zeta\res {E^c})\\
&=&  \hat f_0^{\x_2}\otimes \pi_{\x_2}\zeta=\hat f_0,
\end{eqnarray*}
the same holds for the $y$ marginal.
Moreover:
\begin{eqnarray*}
\int c(x,y)\ d\bar\gamma(x,y)
&=&\int_{E}\int \vert \x_1-\y_1\vert^2\ d\bar\mu^{\x_2}(\x_1)d\zeta(\x_2,\x_2+\z_2)\\& &+ \int_{E^c}\int \vert \x_1-\y_1\vert^2\ d\hat\mu^{\x_2}(\x_1)d\zeta(\x_2,\x_2+\z_2)+\vert \z_2\vert^2\gamma({\hat C}_i)\\
&<& \int_{E}\int \vert \x_1-\y_1\vert^2\ d\hat\mu^{\x_2}(\x_1)d\zeta(\x_2,\x_2+\z_2)\\
&& + \int_{E^c}\int \vert \x_1-\y_1\vert^2\ d\hat \mu^{\x_2}(\x_1)d\zeta(\x_2,\x_2+\z_2)+\vert \z_2\vert^2\gamma({\hat C}_i)\\
&=&\int c(x,y)\ d\hat\gamma(x,y).
\end{eqnarray*}
This is impossible because $\hat\gamma$ is optimal for the transport problem with cost $c$ between its marginals $\hat f_0$ and $\hat f_1$.
\end{proof}

\section{An application to crystalline norms and $L^\infty$ problems}

We present in this section how the results of this paper may be applied, following the strategy presented in \cite{CarDePSan}, to the problem of the existence of optimal transport maps when the cost is a crystalline norm $||x-y||$. We recall that a norm is called crystalline if its unit ball is a convex polytope containing $0$ in its interior with a finite number of faces; In particular, this means that it has the form $||z||=\max_i z\cdot v_i$ for a certain finite set of vectors $(v_i)_{i=1,\dots,k}$. This case has been first studied in \cite{AmbKirPra} and was one of the first steps towards the extension of the Monge problem to general norms. The existence of an optimal map for norm costs has finally been established in full generality in \cite{CD1}.

The strategy proposed in \cite{CarDePSan} concerns the minimization of a transport cost $\int c(x-y)d\pi$ for a convex, but not necessarily strictly convex, function $c$. In the same spirit of the proof of Theorem \ref{THM}, a decomposition may be performed on an optimal transport plan $\gamma$, according to the ``faces'' of the cost. In this case, since the crystalline norm has a finite number of faces $F_i$ (which are the cones on the maximal flat parts of its unit ball), one can write $\gamma=\sum_i \gamma_i$ where the measures $\gamma^i$ are obtained in the following way. If $f_0<<\lcal^d$, duality implies (see  \cite{CarDePSan}) that for a.e. $x$ there is an index $i=i(x)$ such that $(x,y)\in\spt(\gamma)$ implies $x-y\in F_i$ (in case of non-uniqueness of this index just pick one at random, for instance by fixing an order relation on the faces). Then define $\gamma_i=\gamma\res{\{(x,y)\,:\,i(x)=i\}}$.

This means that it is enough to prove that every plan $\gamma^i$ may be actually induced by a transport map (more precisely, that there exists a new plan, with the same marginals and an improved transport cost, which is induced by a map). This is quite easy since the cost, when restricted to a face, is actually affine, and hence every transport plan with the same marginals and supported on $\{(x,y)\,:\,x-y\in F_i\}$ gives the same cost. The issue is thus reduced to finding an arbitrary admissible transport map $T$, with prescribed marginals $f_0^i$ and $f_1^i$ (defined as the marginals of the original $\gamma^i)$, satisfying $x-T(x)\in F_i$. 

Thanks to the results of the present paper this is possible by considering, for instance, the minimization problem (MK) with convex set $C=F_i$. This selects a particular transport map (the one minimizing the quadratic cost among admissible ones), but gives in particular the existence of one such transport map. The only point to be noticed is the fact that a transport plan with finite cost is actually known to exist, and it is $\gamma^i$.

As a consequence, we have a new proof of 
\begin{thm} If $f_0<<dx$ and $||\cdot||$ is a crystalline norm, the transport problem
$$\min\left\{\int||x-y||d\pi\,:\,\pi\in\Pi(f_0,f_1)\right\}$$
admits at least a solution induced by a transport map $T$.
\end{thm}

Analogously, the arguments used in this paper may also be applied to $L^\infty$ problems like those studied in \cite{Winfty}. We already pointed out that the strategy of the quadratic perturbation is useful when one only wants to prove the existence of at least a transport map $T$ sending $f_0$ onto $f_1$ and satisfying $T(x)-x\in C$. Suppose now that we want to solve
$$(M_\infty)\quad\min\left\{\pi-\ess ||y-x||\quad:\quad\pi\in\Pi(f_0,f_1)\right\},$$
where $||\cdot||$ is a (possibly asymmetric) norm whose unit ball is $C$. Notice that we have the equality $\pi-\ess ||y-x||=\max\{||y-x||\,:\,(x,y)\in\spt(\pi)\}$ and that this quantity measures the maximal displacement in terms of the norm $||\cdot||$.
Then it is sufficient to denote by $L$ the minimimal value of such an $\ess$ and to notice that $\pi$ is optimal if and only if it is concentrated on $\{(x,y)\,:\,y-x\in LC\}$. In order to find a solution induced by a transport map $T$ it is enough, for instance, to take a map $T$ solving $(M)$ for the cost $c$ defined with the dilated body $LC$ instead of $C$. Hence we have the following theorem, generalizing the results in \cite{Winfty} (and using, by the way, a different selection principle, since the plan which was selected in \cite{Winfty} was the  infinitely cyclically monotone one, and not that minimizing the quadratic cost).
\begin{thm} Suppose that $C$ satisfies the assumptions of Theorem \ref{THM}, that $0\in {\rm int}(C)$ and that $f_0<<dx$. Let $||\cdot||$ be defined through the formula $||z||:=\inf\{\lambda>0\,:\,z/\lambda\in C\}$. Then, the $L^\infty$ transport problem $(M_\infty)$
admits at least a solution induced by a transport map $T$.
\end{thm}

\end{document}